\documentclass[reqno,11pt]{amsproc}
\usepackage{amsthm,wrapfig,amsmath,amssymb,amsfonts,setspace,verbatim,graphicx,mathtools,mathrsfs,commath,float,mdframed,frame,xcolor,qtree,enumitem, ulem}
\usepackage[hidelinks]{hyperref}
\usepackage[T1]{fontenc}

\newtheorem{ut}{Theorem}
\newtheorem{up}[ut]{Proposition}
\newtheorem{ul}[ut]{Lemma}
\newtheorem*{ucon}{Conjecture}
\newtheorem{uc}[ut]{Corollary}

\theoremstyle{remark}
\newtheorem{ur}{Remark}

\theoremstyle{definition}

\title[dichotomy for spaces  near dimension  $0$]{A dichotomy for spaces  near dimension zero}
\author{David S. Lipham}
\address{Department of Mathematics, Auburn University at Montgomery, Montgomery 
AL 36117, United States of America}
\email{dsl0003@auburn.edu}

\begin{document}

\begin{abstract}We prove that the classes of weakly $1$-dimensional and almost $0$-dimensional spaces are disjoint. The result has applications to  hereditarily locally connected spaces,  $\mathbb R$-trees, and endpoints of smooth fans. \end{abstract}

\maketitle

\section{Introduction}

All spaces under consideration are assumed to be separable and metrizable. For basics of dimension theory and equivalent formulations of dimension, we refer the reader to \cite{eng1}. 

A space $X$ is \textbf{almost zero-dimensional} if $X$ has a neighborhood basis of C-sets, where a \textbf{C-set} is an intersection of clopen subsets of $X$. The dimension of an almost zero-dimensional space is at most $1$ \cite{ov2}. Examples of dimension equal to $1$ include the line-free subgroups of $\ell^2$ known as \textit{Erd\H{o}s spaces} \cite{dim,erd}. They  surface in  homeomorphism groups of manifolds \cite{ov2,dij4,erd}, complex dynamics \cite{31,lipp,lippp}, and projective Fra\"{i}ss\'{e} theory \cite{basso}. Endpoints of smooth fans and $\mathbb R$-trees are universal for almost zero-dimensional spaces \cite{31,erd,ov2}.

 A space $X$ is \textbf{weakly $1$-dimensional} if $\dim(X)=1$ but its dimensional kernel $$\Lambda (X)=\{x\in X:X\text{ is } 1\text{-dimensional at }x\}$$ has dimension $0$. Constructions of weakly $1$-dimensional spaces have been   featured in \cite{w1,cos,inf}; see also the set $P$  in \cite[Figure 1]{liph}.  van Mill and Tuncali proved that a weakly $1$-dimensional space may occur as the set of buried points of a plane continuum \cite{w4}. 


Almost zero-dimensional and weakly $1$-dimensional spaces were studied together in \cite{w3} and were shown to have similar properties.  Both are easily  seen to be \textbf{totally disconnected}, i.e.\ they contain no non-degenerate connected sets.  Arbitrary products of these spaces are at most $1$-dimensional \cite{w3}, and there are universal elements in each class \cite{erd,w2}. 
The primary goal of this paper is to show that the two classes are disjoint.
 
\begin{ut} If $X$ is almost zero-dimensional, then $$\dim(\Lambda(X))=\dim(X).$$Thus there is no weakly $1$-dimensional, almost zero-dimensional space.  \end{ut}

\begin{ur}Many  familiar almost zero-dimensional spaces have the property $\Lambda(X)\in \{\varnothing, X\}$. For such examples Theorem 1 is trivial.  On the other hand, the set of points atop the Fra\"{i}ss\'{e} fence,  denoted $\mathfrak U$ in \cite[Section 5]{basso}, is an almost zero-dimensional Polish space  with the property that $\Lambda(\mathfrak U)$ and $\mathfrak U\setminus \Lambda(\mathfrak U)$ are each dense in $\mathfrak U$ (see  \cite[Proposition 5.19 \& Theorem 5.22]{basso}  and observe that $\Lambda(\mathfrak U)=\mathfrak M$ and $\mathfrak U\setminus \Lambda(\mathfrak U)=\mathfrak U\cap  \mathfrak L$). \end{ur}

We will apply Theorem 1 in hereditarily locally connected spaces, which include all $\mathbb R$-trees. 

A connected space $X$ is \textbf{hereditarily locally connected} if  every connected subspace of $X$ is locally connected.  
An \textbf{$\mathbb R$-tree} is a locally arcwise connected, uniquely arcwise connected space.

 \begin{uc}If $X$ is a hereditarily locally connected space (e.g.\ an $\mathbb R$-tree), then $\dim(\Lambda(Y))=\dim(Y)$ for every subspace $Y$ of $X$. Thus $X$ does not contain a weakly $1$-dimensional space.\end{uc}

Theorem 1 also applies to endpoints of smooth fans. A \textbf{smooth fan} is a compact, connected subset of the Cantor fan (the cone over the Cantor set) \cite{char, eber}. 
The set of all endpoints of a smooth fan $X$ is  denoted $E(X)$. 
 
 \begin{uc}If $X$ is a smooth fan, then $\dim(\Lambda(Y))=\dim(Y)$ for every subspace $Y$ of $E(X)$. Thus $E(X)$ does not contain a weakly $1$-dimensional space. \end{uc}

\section{Almost zero-dimensional versus weakly $1$-dimensional}

In order to prove Theorem 1 we will require a few facts about zero-dimensional spaces and C-sets. 

Recall that a space $X$ is \textbf{zero-dimensional at} $x\in X$ if every neighbourhood of $x$ contains a clopen neighbourhood of $x$.  
  Moreover $X$ is \textbf{zero-dimensional} if it has a neighbourhood basis of clopen sets.

 We will use the symbols $\partial$ and $\overline{\phantom{w}}$ for boundary and closure of a set, respectively. In other texts such as  \cite{eng1},  the boundary is called the `frontier' and is denoted $\text{Fr}$.

\begin{up}[{\cite[Proposition 1.2.12]{eng1}}]If $Z$ is a zero-dimensional subspace of a separable metric space $X$, then for every $x\in X$ and neighborhood $U$ of $x$ there is an open set $V$ containing $x$ such that $V\subset U$ and $\partial V\subset X\setminus Z$.\end{up}

The next proposition is a simple consequence of  {\cite[Theorem 4.19]{erd}}, which states that C-sets in almost zero-dimensional spaces are retracts.

\begin{up}[{\cite[Corollary 4.20]{erd}}] Let $A$ be a C-set in an almost zero-dimensional space $X$. If $B$ is a C-set in $A$, then $B$ is a C-set in $X$.\end{up}

A countable union of C-sets  is called a  \textbf{$\boldsymbol{\sigma}$C-set}.

\begin{up}[{\cite[Theorem 4.4]{coh}}]If $X$ is almost zero-dimensional, then every closed $\sigma\text{C-set}$ in $X$ is a  C-set in $X$.\end{up}

\begin{ul}If $X$ is almost zero-dimensional, then $\Lambda(X)$  is $\sigma\text{$\mathrm{C}$-set}$ in $X$.\end{ul}

\begin{proof}Since $\dim(X)\leq 1$  \cite[Theorem 2]{ov2}, $X$ is zero-dimensional at each point of $X\setminus \Lambda(X)$.  Thus every point of $X\setminus \Lambda(X)$ belongs to
clopen sets in $X$ with arbitrarily small diameters. For every $n\geq 1$ let $\mathscr C_n$ be a covering  of $X\setminus \Lambda(X)$ consisting of clopen subsets of $X$ with diameters $<1/n$. Let  $$A_n=X\setminus \bigcup \mathscr C_n.$$  It is easy to see that  $\Lambda(X)=\bigcup_{n=1}^\infty A_n.$  Clearly $$A_n=\bigcap_{C\in \mathscr C_n} X\setminus C$$ is a C-set in $X$, and the proof is complete. \end{proof}

Now we are prepared to prove the main result.

\begin{proof}[Proof of Theorem 1]Let $X$ be an almost zero-dimensional space. Our aim is to prove $\dim(\Lambda(X))=\dim(X)$. Since  $\dim(X)\leq 1$, it suffices to show that $\dim(\Lambda(X))=0$ implies $\dim(X)=0$.

Suppose  $\dim(\Lambda(X))=0$. Let $x\in X$ and let $U$ be any  neighbourhood of $x$.  By Proposition 4 there is an open set $V\subset U$ such that $x\in V$ and $\partial V\subset \Lambda(X)$.  Let $$A_1,A_2,\ldots$$ be C-sets in $X$ whose union is $\Lambda(X)$,   provided  by Lemma 7. Note that $B_n=\partial V \cap A_n$  is a C-set in the space $A_n$ because $B_n$ closed and $A_n$ is zero-dimensional. So by  Proposition 5, $B_n$ is a C-set in $X$. Hence $$\partial V=\bigcup_{n=1}^\infty B_n$$ is a closed $\sigma\textnormal{C-set}$ in $X$.  By Proposition 6, $\partial V$ is a C-set in $X$. 

Let $W$ be a clopen subset of $X$ which contains $\partial V$ and misses $x$. Then $V\setminus W$ is a clopen subset of $X$ which contains $x$ and lies inside of $U$. Since $x$ and $U$ were arbitrary, this proved  $\dim(X)=0$. \end{proof}

\section{Corollaries}
In order to prove Corollary 2 we will need the following.
\begin{up}[{\cite[Corollary 9]{ov2}}]If $X$ is hereditarily locally connected, then every totally disconnected subset of $X$ is almost zero-dimensional.\end{up}

The proof  of Proposition 8 in \cite{ov2} was based on \cite[Theorem 1.5]{nis} which can be stated as follows:  \textit{In every subspace of a hereditarily locally connected space, each connected component is a  C-set}. An elementary proof of Proposition 8 for $\mathbb R$-trees is given in Appendix A.

\begin{proof}[Proof of Corollary 2]Let $X$ be a hereditarily locally connected space and $Y\subset X$. Since $X$ is $1$-dimensional \cite[Corollary 2.2]{lev}, to prove $\dim(\Lambda(Y))=\dim(Y)$ it suffices to show that $\dim(\Lambda(Y))=0$ implies $\dim(Y)=0$. Suppose $\dim(\Lambda(Y))=0$.  Then $Y$ is totally disconnected, and thus almost zero-dimensional by Proposition 8.   By Theorem 1 we have $\dim(Y)=0$. \end{proof}

\begin{proof}[Proof of Corollary 3]The endpoint set of any smooth fan is almost zero-dimensional because it can be represented as the graph of an upper semi-continuous function with zero-dimensional domain; see \cite[Chapter 4]{erd}, also \cite[Lemma 14]{rod}. So the result is a direct consequence of Theorem 1. \end{proof}

\section{Iteration of $\Lambda$}

Define $$\Lambda^2(X)=\Lambda(\Lambda(X))=\{x\in \Lambda(X):\Lambda(X) \text{ is $1$-dimensional at } x\}.$$ Thus  a  space $X$ is weakly $1$-dimensional if and only if $\dim(X)=1$ and $\Lambda^2(X)=\varnothing$. On the other hand, the proof of Theorem 1 can be localized to show that if $X$ almost zero-dimensional, then $\Lambda^2(X)$ is dense in $\Lambda(X)$. We conjecture the following improvement.

\begin{ucon}If $X$ is almost zero-dimensional, then  $\Lambda^2(X)=\Lambda(X)$.\end{ucon}

\section*{Acknowledgements} I am grateful to Jan J. Dijkstra for past collaborations that influenced this paper, especially for developing the proof of Proposition 6.

\section*{Appendix A}

Below is an elementary proof of Proposition 8 for $\mathbb R$-trees which does not use the hereditarily locally connected property.

\begin{proof}[Proof]Let $X$ be an $\mathbb R$-tree. For any two points $x,y\in X$ let $[x,y]$ denote the unique arc with endpoints $x$ and $y$, and $(x,y]=[x,y]\setminus \{x\}$.

Let $Y\subset X$ be totally disconnected. Let $U$ be any connected open subset of $X$. We will show that $\overline U\cap Y$ is a C-set in $Y$.  Since $X$ is locally connected, this will prove that $Y$ is almost zero-dimensional.

 Let $a\in \overline U$ and $b\in Y\setminus \overline U$. Then $[a,b]\setminus \overline U$ contains an arc, so because $Y$ is totally disconnected  there exists $c\in [a,b]\setminus (\overline U\cup Y)$. Let $A$ be the set of all $x\in X\setminus \{c\}$ such that $[x,c]\cap [a,b]\subset [a,c]$. Clearly $a\in A$ and $b\notin A$. 
 
 We claim that $A$ is clopen in $X\setminus \{c\}$.   To see that $A$ is open, let $x\in A$. Let $V$ be an arcwise connected open set containing $x$ such that $V\cap [c,b]=\varnothing$.  Then each point of $V$ is connected to $c$ by an arc missing $(c,b]$. Thus $V\subset A$. This proves that $A$ is open. To see that $A$ is closed in $X\setminus \{c\}$, let $x\in \overline A\setminus \{c\}$.  Let $y\in (c,b]\setminus \{x\}$.  Let $W$ be arcwise connected open set containing $x$ and not $y$. Since $W$ contains a point of $A$, there is an arc from $x$ to $c$ avoiding $y$. Thus $[x,c]\cap (c,b]\subset \{x\}$. It follows that $x\notin (c,b]$, hence $[x,c]\cap (c,b]=\varnothing$ and $x\in A$. 

By the preceding claim and the fact that $\overline U$ is a connected subset of $X\setminus \{c\}$, we have $\overline U\subset A$. Thus $A\cap Y$ is a relatively clopen subset of $Y$ which contains $\overline U\cap Y$ and misses $b$. Since $b$ was an arbitrary point of $Y\setminus \overline U$, this shows that $\overline U\cap Y$ is a C-set in $Y$ as desired.\end{proof}


\begin{thebibliography}{HD}


\bibitem{basso}G. Basso, R. Camerlo,  Fences, their endpoints, and projective Fra\"{i}ss\'{e} theory. Trans. Amer. Math. Soc. 374 (2021), no. 6, 4501--4535.

\bibitem{char} J. J. Charatonik, On fans, Diss. Math. (Rozprawy Mat.) 54 (1967), 37 p.

\bibitem{dij4}J. J. Dijkstra, On homeomorphism groups of Menger continua, Trans. Amer. Math. Soc. 357 (2005), 2665--2679.

\bibitem{coh}J. J. Dijkstra, D. S. Lipham. On cohesive almost zero-dimensional spaces. Canad. Math. Bull. 64 (2021), no. 2, 429--441. 

\bibitem{erd}J. J. Dijkstra, J. van Mill, Erd\H{o}s space and homeomorphism groups of manifolds, Mem. Amer. Math. Soc. 208 (2010), no. 979.

\bibitem{cos} A. Dow, K. P. Hart, Cosmic dimensions. Topology Appl. 154 (2007), no. 12, 2449--2456.

\bibitem{eber}C. Eberhart, A note on smooth fans, Colloq. Math. 20 (1969), 89--90.

\bibitem{eng1}R. Engelking, Dimension Theory, Volume 19 of Mathematical Studies, North-Holland, Amsterdam, 1978.

\bibitem{dim}P. Erd\H{o}s, The dimension of the rational points in Hilbert space, Ann. of Math. (2) 41 (1940), 734--736.

\bibitem{rod} R. Hernández-Gutiérrez, L. C. Hoehn, Smooth fans that are endpoint rigid.  \url{https://arxiv.org/abs/2206.12776}. preprint (2022).

\bibitem{31}K. Kawamura, L. G. Oversteegen,  E. D. Tymchatyn, On homogeneous totally disconnected $1$-dimensional spaces, Fund. Math. 150 (1996), 97--112.
\bibitem{lev}M. Levin, E. D.  Tymchatyn,  On the dimension of almost n-dimensional spaces. Proc. Amer. Math. Soc. 127 (1999), no. 9, 2793--2795.

\bibitem{lipp}D. S. Lipham, Erd\H{o}s space in Julia sets, \url{https://arxiv.org/pdf/2004.12976.pdf}.  preprint (2022).

\bibitem{lippp} D. S. Lipham, Distinguishing endpoint sets from Erdős space. Math. Proc. Cambridge Philos. Soc. 173 (2022), no. 3, 635--646.

\bibitem{liph}D. S. Lipham, Totally disconnected subsets of chainable continua, Topology Appl. 317 (2022), Article 108187, 1--4.

\bibitem{nis}T. Nishiura, E. D. Tymchatyn, Hereditarily locally connected spaces, Houston J. Math. 2 (1976),  581--599.

\bibitem{ov2}L. G. Oversteegen, E. D.  Tymchatyn. On the dimension of certain totally disconnected spaces, Proc. Amer. Math. Soc. 122 (1994), 885--891.

\bibitem{inf}J. van Mill, The infinite-dimensional topology of function spaces. North-Holland Mathematical Library, 64. North-Holland Publishing Co., Amsterdam, 2001.

 \bibitem{w1}J. van Mill, R. Pol,  On the existence of weakly n-dimensional spaces. Proc. Amer. Math. Soc. 113 (1991), no. 2, 581--585.

\bibitem{w2}J. van Mill, R. Pol,  Note on weakly n-dimensional spaces. Monatsh. Math. 132 (2001), no. 1, 25--33.


 \bibitem{w3}J. van Mill, R. Pol,  On spaces without non-trivial subcontinua and the dimension of their products. Topology Appl. 142 (2004), no. 1-3, 31--48. 
 
 \bibitem{w4} J. van Mill, M. Tuncali, Plane continua and totally disconnected sets of buried points. Proc. Amer. Math. Soc. 140 (2012), no. 1, 351--356.


\end{thebibliography}
\end{document}